\documentclass{amsart}
\usepackage{amsmath}
\usepackage{amssymb}
\usepackage{amsfonts}
\usepackage{hyperref}

\newtheorem{theorem}{Theorem}[section]
\theoremstyle{plain}

\newtheorem{corollary}{Corollary}[section]

\newtheorem{remark}{Remark}

\numberwithin{equation}{section}

\begin{document}

 \title[A companion of Ostrowski like inequality ]{A companion of Ostrowski like inequality for mappings whose second derivatives belong to $L^{\infty}$ spaces and applications}
\author[W. J. Liu]{Wenjun Liu}
\address[W. J. Liu]{College of Mathematics and Statistics\\
Nanjing University of Information Science and Technology \\
Nanjing 210044, China}
\email{wjliu@nuist.edu.cn}



\begin{abstract}
A companion of Ostrowski like inequality
for mappings whose second derivatives belong to $L^{\infty}$ spaces is established. Applications to
composite quadrature rules, and to probability density functions are also given.
\end{abstract}

\maketitle

\section{Introduction}

In 1938, Ostrowski established the following interesting integral
inequality (see \cite{mpf1993}) for differentiable mappings with bounded
derivatives:
\begin{theorem}\label{Th1.1}
Let $f:[a,b]\rightarrow\mathbb{R}$ be a differentiable mapping on
$(a,b)$ whose derivative is bounded on $(a,b)$ and denote
$\|f'\|_{\infty}=\displaystyle{\sup_{t\in(a,b)}}|f'(t)|<\infty$.
Then for all $x\in[a,b]$ we have
\begin{equation}
\left|f(x)-\frac{1}{b-a}\int_{a}^{b}f(t)dt\right|\leq\left[\frac{1}{4}
+\frac{(x-\frac{a+b}{2})^{2}}{(b-a)^{2}}\right](b-a)\|f'\|_{\infty}.
\end{equation}\label{1.1}
The constant $\frac{1}{4}$ is sharp in the sense that it can not be
replaced by a smaller one.
\end{theorem}

Motivated by \cite{gs2002}, Dragomir \cite{d2005} proved some companions of Ostrowski's
inequality, as follows:
\begin{theorem}\label{Th1.2}
Let $f:[a,b]\rightarrow\mathbb{R}$ be an absolutely continuous
function on $[a,b]$. Then the following inequalities
\begin{align}
&\left|\frac{f(x)+f(a+b-x)}{2}-\frac{1}{b-a}\int_{a}^{b}f(t)dt\right|
\nonumber
\\
\leq&\left\{
{\begin{array}{l}\left[\frac{1}{8}+2\left(\frac{x-\frac{3a+b}{4}}{b-a}\right)^{2}\right]
(b-a)\|f'\|_{\infty},\quad f'\in L^{\infty}[a,b], \\
\frac{2^{1/q}}{(q+1)^{1/q}}\left[\left(\frac{x-a}{b-a}\right)^{q+1}+
\left(\frac{\frac{a+b}{2}-x}{b-a}\right)^{q+1}\right]^{1/q}(b-a)^{1/q}\|f'\|_{p},\\
\quad \quad \quad\quad\quad
p>1, \frac{1}{p}+\frac{1}{q}=1\quad \text{and}\quad f'\in L^{p}[a,b], \\
\left[\frac{1}{4}+\left|\frac{x-\frac{3a+b}{4}}{b-a}\right|\right]\|f'\|_{1},
\quad\quad\quad \quad\quad\quad f'\in L^{1}[a,b] \\
\end{array}}\right. \label{1.2}
\end{align}
hold for all $x\in[a,\frac{a+b}{2}]$.
\end{theorem}

Recently, Alomari \cite{a2011} introduced a companion of Dragomir's
generalization of Ostrowsk's inequality for absolutely continuous
mappings whose first derivatives are belong to $L^{\infty}([a,b])$.

\begin{theorem}\label{Th1.3}
Let $f:[a,b]\rightarrow\mathbb{R}$ be an absolutely continuous
mappings on $(a,b)$ whose derivative is bounded on $[a,b]$. Then the
inequality
\begin{align}
&\left|\left[(1-\lambda)\frac{f(x)+f(a+b-x)}{2}+\lambda\frac{f(a)+f(b)}{2}\right]-
\frac{1}{b-a}\int_{a}^{b}f(t)dt\right| \nonumber \\
\leq&\left[\frac{1}{8}(2\lambda^{2}+(1-\lambda)^{2})+
2\frac{\left(x-\frac{(3-\lambda)a+(1+\lambda)b}{4}\right)^{2}}{(b-a)^{2}}\right](b-a)\|f'\|_{\infty}
\label{1.3}
\end{align}
holds for all $\lambda\in[0,1]$ and $x\in[a+\lambda\frac{b-a}{2},
\frac{a+b}{2}]$.
\end{theorem}

In \eqref{1.3}, choose $\lambda=\frac{1}{2}$, one can get
\begin{align}
&\left|\frac{1}{2}\left[\frac{f(x)+f(a+b-x)}{2}+\frac{f(a)+f(b)}{2}\right]
-\frac{1}{b-a}\int_{a}^{b}f(t)dt\right| \nonumber \\
\leq&\left[\frac{3}{32}+2\frac{(x-\frac{5a+3b}{8})^{2}}{(b-a)^{2}}\right](b-a)\|f'\|_{\infty}.
\label{1.4}
\end{align}
And if choose $x=\frac{a+b}{2}$, then one has
\begin{align}
\left|\frac{1}{2}\left[f\left(\frac{a+b}{2}\right)+\frac{f(a)+f(b)}{2}\right]
-\frac{1}{b-a}\int_{a}^{b}f(t)dt\right|\leq\frac{1}{8}(b-a)\|f'\|_{\infty}.
\label{1.5}
\end{align}
It's shown in \cite{a2011} that the constant $\frac{1}{8}$ is the best
possible.

In related work, Dragomir and Sofo \cite{ds2000}  developed the following
Ostrowski like integral inequality for twice differentiable mapping.

\begin{theorem}\label{Th1.4}
Let $f:[a,b]\rightarrow\mathbb{R}$ be a mapping whose first
derivative is absolutely continuous on $[a,b]$ and assume that the
second derivative $f''\in L^{\infty}([a,b])$. Then we have the
inequality
\begin{align}
&\left|\frac{1}{2}\left[f(x)+\frac{f(a)+f(b)}{2}\right]-
\frac{1}{2}\left(x-\frac{a+b}{2}\right)f'(x)-\frac{1}{b-a}\int_{a}^{b}f(t)dt\right| \nonumber \\
\leq&\left[\frac{1}{48}+\frac{1}{3}\frac{|x-\frac{a+b}{2}|^{3}}{(b-a)^{3}}\right]
(b-a)^{2}\|f''\|_{\infty},\label{1.6}
\end{align}
for all $x\in[a,b]$.
\end{theorem}

 In \eqref{1.6}, the authors pointed out that the midpoint
 $x=\frac{a+b}{2}$ gives the best estimator, i.e.,
\begin{align}
\left|\frac{1}{2}\left[f\left(\frac{a+b}{2}\right)+\frac{f(a)+f(b)}{2}\right]
-\frac{1}{b-a}\int_{a}^{b}f(t)dt\right|\leq\frac{1}{48}(b-a)^{2}\|f''\|_{\infty}.\label{1.7}
\end{align}
In fact, we can choose  $f (t) = (t-a)^2$ in \eqref{1.7} to  prove that the constant $\frac{1}{48}$ in inequality \eqref{1.7} is sharp.

For other related results, the reader may be refer to \cite{a20111, a20112, a20113, bdg2009, d20051,  d2011, d2001, hn2011, l2009, s20101, s2010, ss2011, thd2008, thyc2011, v2011} and the
references therein.

Motivated by previous works \cite{a2011, d2002, d2005, ds2000}, we investigate in this paper a
companion of the above mentioned Ostrowski like integral inequality
for twice differentiable mappings. Our result gives a smaller
estimator than \eqref{1.7} (see \eqref{2.9} below). Some
applications to
composite quadrature rules, and to probability density functions are also given.

\section{A companion of Ostrowski like inequality}

The following companion of Ostrowski like inequality holds:

\begin{theorem}\label{Th2.1}
Let $f:[a,b]\rightarrow\mathbb{R}$ be a mapping whose first
derivative is absolutely continuous on $[a,b]$ and assume that the
second derivative $f''\in L^{\infty}([a,b])$. Then we have the
inequality
\begin{align}
&\left|\frac{1}{2}\left[\frac{f(x)+f(a+b-x)}{2}+\frac{f(a)+f(b)}{2}\right]\right.\nonumber \\
&\left.-\frac{1}{2}\left(x-\frac{a+b}{2}\right)\frac{f'(x)-f'(a+b-x)}{2}
-\frac{1}{b-a}\int_{a}^{b}f(t)dt\right| \nonumber \\
\leq&\left[\frac{1}{3}\frac{(\frac{a+3b}{4}-x)(x-a)^{2}}{(b-a)^{3}}+
\frac{1}{3}\frac{(\frac{a+b}{2}-x)^{3}}{(b-a)^{3}}\right](b-a)^{2}\|f''\|_{\infty}\label{2.1}
\end{align}
for all $x\in[a,\frac{a+b}{2}]$. The first constant $\frac{1}{3}$ in the right hand side of \eqref{2.1}  is sharp in the sense that it can not be
replaced by a smaller one provided that $x\neq \frac{a+3b}{4}$ and $x\neq a$.
\end{theorem}
\begin{proof}
Define the kernel $K(t):[a,b]\rightarrow\mathbb{R}$ by
\begin{align}
K(t):=\left\{ {\begin{array}{l} t-a, \quad \quad t\in[a,x], \\
t-\frac{a+b}{2}, \quad t\in(x,a+b-x], \\
t-b, \quad \quad t\in(a+b-x,b],\\
\end{array}}\right. \label{2.2}
\end{align}
for all $x\in[a,\frac{a+b}{2}]$. Integrating by parts, we obtain
(see  \cite{d2005})
\begin{align}
\frac{1}{b-a}\int_{a}^{b}K(t)g'(t)dt=\frac{g(x)+g(a+b-x)}{2}-\frac{1}{b-a}\int_{a}^{b}g(t)dt.
\label{2.3}
\end{align}
Now choose in \eqref{2.3}, $g(x)=(x-\frac{a+b}{2})f'(x)$, to get
\begin{align}
&\frac{1}{b-a}\int_{a}^{b}K(t)\left[f'(t)+\left(t-\frac{a+b}{2}\right)f''(t)\right]dt \nonumber \\
=&\frac{1}{2}\left(x-\frac{a+b}{2}\right)\left[f'(x)-f'(a+b-x)\right]-\frac{1}{b-a}
\int_{a}^{b}\left(t-\frac{a+b}{2}\right)f'(t)dt. \label{2.4}
\end{align}
Integrating by parts, we have
\begin{align}
\frac{1}{b-a}
\int_{a}^{b}\left(t-\frac{a+b}{2}\right)f'(t)dt=\frac{f(a)+f(b)}{2}-\frac{1}{b-a}
\int_{a}^{b}f(t)dt. \label{2.5}
\end{align}
Also upon using \eqref{2.3}, we get
\begin{align}
&\frac{1}{b-a}\int_{a}^{b}K(t)\left[f'(t)+\left(t-\frac{a+b}{2}\right)f''(t)\right]dt \nonumber \\
=&\frac{1}{b-a}\int_{a}^{b}K(t)f'(t)dt+\frac{1}{b-a}
\int_{a}^{b}K(t)\left(t-\frac{a+b}{2}\right)f''(t)dt  \nonumber \\
=&\frac{f(x)+f(a+b-x)}{2}-\frac{1}{b-a}\int_{a}^{b}f(t)dt+\frac{1}{b-a}
\int_{a}^{b}K(t)\left(t-\frac{a+b}{2}\right)f''(t)dt. \label{2.6}
\end{align}
It follows from \eqref{2.4}--\eqref{2.6} that
\begin{align}
&\frac{1}{2(b-a)}
\int_{a}^{b}K(t)\left(t-\frac{a+b}{2}\right)f''(t)dt \nonumber \\
=&\frac{1}{b-a}\int_{a}^{b}f(t)dt-\frac{1}{2}\left[\frac{f(x)+f(a+b-x)}{2}+\frac{f(a)+f(b)}{2}\right] \nonumber \\
&+\frac{1}{2}\left(x-\frac{a+b}{2}\right)\frac{f'(x)-f'(a+b-x)}{2}.\label{2.7}
\end{align}
Now using \eqref{2.7} we obtain
\begin{align}
&\left|\frac{1}{2}\left[\frac{f(x)+f(a+b-x)}{2}+\frac{f(a)+f(b)}{2}\right]\right. \nonumber\\
&\left. -\frac{1}{2}\left(x-\frac{a+b}{2}\right)\frac{f'(x)-f'(a+b-x)}{2}-\frac{1}{b-a}\int_{a}^{b}f(t)dt\right| \nonumber \\
\leq&\frac{\|f''\|_{\infty}}{2(b-a)}\int_{a}^{b}|K(t)|\left|t-\frac{a+b}{2}\right|dt.
\label{2.8}
\end{align}

Since $x\in[a,\frac{a+b}{2}]$, we have
\begin{align*}
I:=&\int_{a}^{b}|K(t)|\left|t-\frac{a+b}{2}\right|dt  \\
=&\int_{a}^{x}(t-a)\left|t-\frac{a+b}{2}\right|dt+\int_{x}^{a+b-x}\left(t-\frac{a+b}{2}\right)^{2}dt
+\int_{a+b-x}^{b}(b-t)\left|t-\frac{a+b}{2}\right|dt  \\
=&\int_{a}^{x}(t-a)\left(\frac{a+b}{2}-t\right)dt+\int_{x}^{a+b-x}\left(t-\frac{a+b}{2}\right)^{2}dt
+\int_{a+b-x}^{b}(b-t)\left(t-\frac{a+b}{2}\right)dt \\
=&\frac{(a+3b-4x)(x-a)^{2}}{12}+\frac{2}{3}\left(\frac{a+b}{2}-x\right)^{3}+\frac{(a+3b-4x)(x-a)^{2}}{12}
\\
=&\frac{(a+3b-4x)(x-a)^{2}}{6}+\frac{2}{3}\left(\frac{a+b}{2}-x\right)^{3},
\end{align*}
and referring to \eqref{2.8}, we obtain the result \eqref{2.1}.

The sharpness of the constant $\frac{1}{3}$ can be proved in a special case for  $x=\frac{a+b}{2}$ (see the line behind \eqref{1.7}).
\end{proof}

\begin{remark} \label{re1}
If we take $x=\frac{a+b}{2}$ in \eqref{2.1}, we recapture the sharp inequality \eqref{1.7}.
If we take $x=a$ in \eqref{2.1}, we obtain the perturbed trapezoid type inequality
$$
\left|\frac{f(a)+f(b)}{2}-\frac{b-a}{8}[f'(b)-f'(a)]
-\frac{1}{b-a}\int_{a}^{b}f(t)dt\right|\leq\frac{(b-a)^{2}}{24}\|f''\|_{\infty},
$$
which has a smaller estimator than the sharp trapezoid  inequality
$$
\left|\frac{f(a)+f(b)}{2}
-\frac{1}{b-a}\int_{a}^{b}f(t)dt\right|\leq\frac{(b-a)^{2}}{8}\|f''\|_{\infty}
$$
stated in \cite[Proposition 2.7]{d2001}.\end{remark}

\begin{remark} \label{re2}
Consider
\begin{align*}
F(x)=\left(\frac{a+3b}{4}-x\right)(x-a)^{2}+\left(\frac{a+b}{2}-x\right)^{3}
\end{align*}
for $x\in[a,\frac{a+b}{2}]$. It's easy to know that $F(x)$ obtains
its minimal value at $x=\frac{3a+b}{4}$. Therefore, in \eqref{2.1},
the point $x=\frac{3a+b}{4}$ gives the best estimator, i.e.,
\begin{align}
&\left|\frac{1}{2}\left[\frac{f(\frac{3a+b}{4})+f(\frac{a+3b}{4})}{2}
+\frac{f(a)+f(b)}{2}\right]+\frac{b-a}{8}\frac{f'(\frac{3a+b}{4})
-f'(\frac{a+3b}{4})}{2}-\frac{1}{b-a}\int_{a}^{b}f(t)dt\right|
\nonumber \\
\leq&\frac{1}{64}(b-a)^{2}\|f''\|_{\infty}, \label{2.9}
\end{align}
the right hand side of which is smaller than that of \eqref{1.7}.
\end{remark}

\section{Application to Composite Quadrature Rules}

In \cite{ds2000}, the authors utilized inequality \eqref{1.6} to give
estimates of composite quadrature rules which was pointed out have a
markedly smaller error than that which may be obtained by the
classical results. In this section, we apply our previous inequality
\eqref{2.1} to give us estimates of new composite quadrature rules
which have a further smaller error.

\begin{theorem}\label{Th3.1}
Let $I_{n}: a=x_{0}<x_{1}<\cdot\cdot\cdot<x_{n-1}<x_{n}=b$ be a
partition of the interval $[a,b]$, $h_{i}=x_{i+1}-x_{i}$,
$\nu(h):=\max\{h_{i}:i=1,\cdot\cdot\cdot,n\}$,
$\xi_{i}\in[x_{i},\frac{x_{i}+x_{i+1}}{2}]$, and
\begin{align*}
S(f,I_{n},\xi)=&\frac{1}{4}\sum_{i=0}^{n-1}\left[f(x_{i})+f(\xi_{i})+f(x_{i}+x_{i+1}-\xi_{i})+f(x_{i+1})\right]h_{i}
\\
&-\frac{1}{4}\sum_{i=0}^{n-1}h_{i}\left(\xi_{i}-\frac{x_{i}+x_{i+1}}{2}\right)\left[f'(\xi_{i})
-f'(x_{i}+x_{i+1}-\xi_{i})\right],
\end{align*}
then
\begin{align*}
\int_{a}^{b}f(x)dx=S(f,I_{n},\xi)+R(f,I_{n},\xi)
\end{align*}
and the remainder $R(f,I_{n},\xi)$ satisfies the estimate
\begin{align}
|R(f,I_{n},\xi)|\leq\frac{1}{3}\|f''\|_{\infty}\left[\sum_{i=0}^{n-1}\left(\frac{x_{i}
+3x_{i+1}}{4}-\xi_{i}\right)(x_{i}-\xi_{i})^{2}+\sum_{i=0}^{n-1}\left(\frac{x_{i}
+x_{i+1}}{2}-\xi_{i}\right)^{3}\right]. \label{3.1}
\end{align}
\end{theorem}
\begin{proof}
Inequality \eqref{2.1} can be written as
\begin{align}
&\left|\int_{a}^{b}f(t)dt-\frac{1}{4}\left[f(a)+f(x)+f(a+b-x)+f(b)\right](b-a)\right.\nonumber\\
&\left. +
\frac{b-a}{4}\left(x-\frac{a+b}{2}\right)[f'(x)-f'(a+b-x)]
\right| \nonumber \\
\leq&\frac{1}{3}\left[\left(\frac{a+3b}{4}-x\right)(x-a)^{2}+
\left(\frac{a+b}{2}-x\right)^{3}\right]\|f''\|_{\infty}.\label{3.2}
\end{align}
Applying \eqref{3.2} on $\xi_{i}\in[x_{i},\frac{x_{i}+x_{i+1}}{2}]$,
we have
\begin{align*}
&\left|\int_{x_{i}}^{x_{i+1}}f(t)dt-\frac{1}{4}\left[f(x_{i})+f(\xi_{i})+
f(x_{i}+x_{i+1}-\xi_{i})+f(x_{i+1})\right]h_{i}\right. \\
&+\left.\frac{h_{i}}{4}\left(\xi_{i}
-\frac{x_{i}+x_{i+1}}{2}\right)\left[f'(\xi_{i})
-f'(x_{i}+x_{i+1}-\xi_{i})\right]\right| \\
\leq&\frac{1}{3}\left[\left(\frac{x_{i}+3x_{i+1}}{4}-\xi_{i}\right)(x_{i}-\xi_{i})^{2}
+\left(\frac{x_{i}+x_{i+1}}{2}-\xi_{i}\right)^{3}\right]\|f''\|_{\infty}.
\end{align*}
Now summing over $i$ from $0$ to $n-1$ and utilizing the triangle
inequality, we have
\begin{align*}
&\left|\int_{a}^{b}f(t)dt-S(f,I_{n},\xi)\right| \\
=&\left|\sum_{i=0}^{n-1}\int_{x_{i}}^{x_{i+1}}f(t)dt-\frac{1}{4}\sum_{i=0}^{n-1}\left[f(x_{i})+f(\xi_{i})+
f(x_{i}+x_{i+1}-\xi_{i})+f(x_{i+1})\right]h_{i}\right. \\
&+\left.\frac{1}{4}\sum_{i=0}^{n-1}h_{i}\left(\xi_{i}
-\frac{x_{i}+x_{i+1}}{2}\right)\left[f'(\xi_{i})
-f'(x_{i}+x_{i+1}-\xi_{i})\right]\right| \\
\leq&\frac{1}{3}\|f''\|_{\infty}\sum_{i=0}^{n-1}
\left[\left(\frac{x_{i}+3x_{i+1}}{4}-\xi_{i}\right)(x_{i}-\xi_{i})^{2}
+\left(\frac{x_{i}+x_{i+1}}{2}-\xi_{i}\right)^{3}\right]
\end{align*}
and therefore \eqref{3.1} holds.
\end{proof}

\begin{corollary}
 If we choose $\xi_{i}=\frac{3x_{i}+x_{i+1}}{4}$,
then we have
\begin{align*}
\overline{S}(f,I_{n})=&\frac{1}{4}\sum_{i=0}^{n-1}\left[f(x_{i})+f\left(\frac{3x_{i}+x_{i+1}}{4}\right)
+f\left(\frac{x_{i}+3x_{i+1}}{4}\right)+f(x_{i+1})\right]h_{i}\\
&+\frac{1}{16}\sum_{i=0}^{n-1} \left[f'\left(\frac{3x_{i}+x_{i+1}}{4}\right)
-f'\left(\frac{x_{i}+3x_{i+1}}{4}\right)\right]h_{i}^{2}
\end{align*}
and
\begin{align}
 |\overline{R}(f,I_{n})|\leq\frac{1}{64}\|f''\|_{\infty}\sum_{i=0}^{n-1}h_{i}^{3}.
 \label{3.3}
\end{align}
\end{corollary}

\begin{remark} \label{re2}
It is obvious that inequality \eqref{3.3} is better than   \cite[inequality \eqref{3.1}]{ds2000} due to a smaller error.
\end{remark}

\section{Application to probability density functions}

Now, let $X$ be a random variable taking values in the finite interval
$[a,b]$, with the probability density function $f : [a, b]\rightarrow [0, 1]$ and with
the cumulative distribution function $$F (x) = Pr (X \leq  x) = \int_a^x f (t) dt.$$

The following result holds:
\begin{theorem}\label{Th4.1}
With the above assumptions, we have the inequality
\begin{align}
&\left|\frac{1}{2}\left[\frac{F(x)+F(a+b-x)}{2}+\frac{1}{2}\right]\right.\nonumber \\
&\left.-\frac{1}{2}\left(x-\frac{a+b}{2}\right)\frac{f(x)-f(a+b-x)}{2}
-\frac{b-E(X)}{b-a} \right| \nonumber \\
\leq&\left[\frac{(\frac{a+3b}{4}-x)(x-a)^{2}}{3(b-a)^{3}}+
\frac{1}{3}\frac{(\frac{a+b}{2}-x)^{3}}{(b-a)^{3}}\right](b-a)^{2}\|f'\|_{\infty}\label{4.1}
\end{align}
for all $x\in[a,\frac{a+b}{2}]$, where $E (X)$ is the expectation of $X$.
\end{theorem}
\begin{proof}  Follows by \eqref{2.1} on choosing $f = F$ and taking into account
$$E(X)=\int_a^b t dF(t)=b-\int_a^b F(t)dt,$$
we obtain \eqref{4.1}.
\end{proof}

In particular, we have:

\begin{corollary}\label{Th4.1}
With the above assumptions, we have the inequality
\begin{align*}
&\left|\frac{1}{2}\left[\frac{F(\frac{3a+b}{4})+F(\frac{a+3b}{4})}{2}
+\frac{1}{2}\right]+\frac{b-a}{8}\frac{f(\frac{3a+b}{4})
-f(\frac{a+3b}{4})}{2}-\frac{b-E(X)}{b-a} \right|
\nonumber \\
\leq&\frac{1}{64}(b-a)^{2}\|f'\|_{\infty}.
\end{align*}
\end{corollary}

\subsection*{Acknowledgments}
This work was partly supported by the National Natural Science Foundation
of China (Grant No. 40975002) and the Natural Science Foundation of the Jiangsu
Higher Education Institutions (Grant No. 09KJB110005). The author would like to thank Professor J. Duoandikoetxea   for bringing reference \cite{d2001} to his attention.

\end{document}